\documentclass[11pt]{amsart}  
\usepackage{amssymb,amsmath,amsthm,amscd}
\usepackage{graphicx}
\usepackage[hypertex]{hyperref}

\numberwithin{equation}{section}

\newtheoremstyle{fancy1}{10pt}{10pt}{\itshape}{12pt}{\textsc\bgroup}{.\egroup}{8pt}{
}
\newtheoremstyle{fancy2}{10pt}{10pt}{}{12pt}{\itshape}{.}{8pt}{ }

\theoremstyle{fancy1}

\newtheorem{cor}[equation]{Corollary}
\newtheorem{lem}[equation]{Lemma}
\newtheorem{prop}[equation]{Proposition}

\newtheorem*{thm*}{Theorem}

\newtheorem{main}{Theorem}
\newtheorem*{main*}{Theorem}

\newtheorem*{cor*}{Corollary}
\newtheorem*{prop*}{Proposition}

\newtheorem*{problem*}{Problem}

\theoremstyle{fancy2}
\newtheorem{definition}[equation]{Definition}

\newtheorem*{rems*}{Remarks}

\newtheorem*{rem*}{Remark}
\newtheorem{example}{Example}
\newtheorem*{example*}{Example}

\newcommand{\cref}[1]{Corollary~\ref{#1}}
\newcommand{\dref}[1]{Definition~\ref{#1}}

\newcommand{\lref}[1]{Lemma~\ref{#1}}
\newcommand{\pref}[1]{Proposition~\ref{#1}}

\newcommand{\gt}{\theta}
\newcommand{\e}{\epsilon}

\newcommand{\CP}{\mathbb{C\mkern1mu P}}

\newcommand{\Sph}{\mathbb{S}}
\newcommand{\Disc}{\mathbb{D}}

\newcommand{\C}{{\mathbb{C}}}
\newcommand{\R}{{\mathbb{R}}}
\newcommand{\Z}{{\mathbb{Z}}}

\newcommand{\G}{\ensuremath{\operatorname{G}}}

\newcommand{\SO}{\ensuremath{\operatorname{SO}}}

\renewcommand{\S}{\ensuremath{\operatorname{S}}}

\newcommand{\fg}{{\mathfrak{g}}}
\newcommand{\fk}{{\mathfrak{k}}}
\newcommand{\fh}{{\mathfrak{h}}}
\newcommand{\fm}{{\mathfrak{m}}}

\newcommand{\fp}{{\mathfrak{p}}}

\def\con#1=#2(#3){#1 \equiv #2 \bmod{#3}}

\newcommand{\ml}{\langle}                     
\newcommand{\mr}{\rangle}                    

\newcommand{\tr}{\ensuremath{\operatorname{tr}}}

\renewcommand{\Im}{\ensuremath{\operatorname{Im}}}
\newcommand{\Ad}{\ensuremath{\operatorname{Ad}}}

\renewcommand{\sec}{\ensuremath{\operatorname{sec}}}
\newcommand{\Ric}{\ensuremath{\operatorname{Ric}}}

\DeclareMathOperator{\Id}{Id}

\DeclareMathOperator{\spam}{span}

\newcommand{\Kpo}{K_{\scriptscriptstyle{0}}^{\scriptscriptstyle{+}}}
\newcommand{\Kmo}{K_{\scriptscriptstyle{0}}^{\scriptscriptstyle{-}}}
\newcommand{\Kpm}{K^{\scriptscriptstyle{\pm}}}
\newcommand{\Kp}{K^{\scriptscriptstyle{+}}}
\newcommand{\Km}{K^{\scriptscriptstyle{-}}}
\newcommand{\Ko}{K_{\scriptscriptstyle{0}}}

\newcommand{\subo}{_{\scriptscriptstyle{0}}}

\newcommand{\coo}{{cohomogeneity one}}

\begin{document}

\title{Concavity and rigidity in non-negative curvature}

\dedicatory{ Dedicated to D.V.~Alekseevsky on his 70th birthday }

\author{Luigi Verdiani}
\address{University of Firenze}
\email{verdiani@math.unifi.it}
\author{Wolfgang Ziller}
\address{University of Pennsylvania}
\email{wziller@math.upenn.edu}
\thanks{ The first named author was supported by the University of Pennsylvania
and by IMPA and would like to thank both Institutes for their
hospitality.   The second named author was supported by a grant from
the National Science Foundation, the Max Planck Institute in Bonn,
CAPES and IMPA}

\maketitle

\begin{abstract}\noindent
We show that for a manifold with non-negative curvature one obtains
a collection of concave functions, special cases of which are the
concavity of the length of a Jacobi field in dimension 2, and the
concavity of the volume in general. We use these functions to show
that there are many cohomogeneity one manifolds which do not carry
an analytic invariant metric with non-negative curvature. This
implies in particular, that  one of the candidates in \cite{GWZ}
does not carry an  invariant metric with positive curvature.
\end{abstract}


\bigskip

There are few known examples of manifolds with positive sectional curvature in Riemannian geometry.
Until recently, they were all homogeneous spaces \cite{Be,Wa,AW} and biquotients \cite{Es1,Es2,Ba},
 i.e., quotients  of compact Lie groups $G$ by a free isometric ``two sided" action of a subgroup $H \subset G
\times G$.  See \cite{Z1} for a survey of the known examples.
Recently a new
    example of a positively curved 7-manifold, homeomorphic but not
diffeomorphic to $T_1\Sph^4$, was constructed in \cite{GVZ}, see
also \cite{De} for a different approach. A new method has also been
proposed in  \cite{PW} to construct a metric of positive curvature
on the Gromoll-Meyer exotic 7-sphere. The new example in \cite{GVZ}
is part of a larger family of ``candidates" for positive curvature
discovered in \cite{GWZ}. One of the applications of this paper is
to exclude one of these candidates.

\smallskip

The obstruction that we use to do this turns out to be of a general
nature that does not require the presence of a group action. It
comes from a new concavity property of Jacobi fields in positive
curvature. The method also gives rise to certain rigidity properties
in nonnegative curvature.

\smallskip

Let $c(t)$ be a geodesic in $M^{n+1}$ and $J(t)$ a Jacobi field
along $c$. For a surface it is well known that positive curvature is
equivalent to requiring that the length of all Jacobi fields is
strictly concave. In higher dimensions, the length $|J|$ satisfies
the differential equation
 $$
\frac{|J|''}{|J|}=-\sec_M(\dot{c},J)+\frac{|J'|^2 }{|J|^2 }\sin^2( \sphericalangle(J^{\; '},J)).
 $$
 Thus in  negative
 curvature  $|J|$ is a strictly convex function.  But in positive curvature
 $|J|$ does not have any distinctive properties. For example,
 the Hopf action on a round sphere induces a Killing vector field of
 constant length.

 \smallskip

 For positive curvature we suggest the concept of a ``virtual"
 Jacobi field. For this it is best to study Jacobi fields via Jacobi tensors. Let $A_t$ be a solution of the
 differential equation
 $$
A''+RA=0
 $$
where  $E_t=\dot{c}(t)^\perp\subset T_{c(t)}M$ and, after a choice
of a base point $t_0$,  $A_t\colon E_{t_0}\to E_t$ and
$R=R(\cdot,\dot{c})\dot{c}\colon E_t\to E_t$. $A$ is uniquely
determined by $A_{t_0}$ and $A'_{t_0}$. Thus for any $v\in E_{t_0}$,
$J(t)=A_tv$ is a Jacobi field along $c$. We denote by $A^*$ the
adjoint of $A$ and call a point $c(t^*)$ regular if $A_{t^*}$ is
invertible. The Jacobi tensor $A$ is called a Lagrange tensor if $A$
is non-degenerate (i.e. Av is not the 0-Jacobi field  for all $v$)
and $S:=A'A^{-1}$ is symmetric at regular points. Equivalently, $S$
is the shape operator of a family of parallel hypersurfaces
orthogonal to $c$.

\begin{main} Let $A$ be a Lagrange tensor along the geodesic $c$ and $v\in E_{t_0}$ non-zero.
 Define $Z_t=(A_t^*)^{-1}v$ and let $g=g_v(t)=\dfrac{||v||^2}{||Z_t||}$. Then
\begin{enumerate}
\item[(a)] $g_v(t)\le ||A_tv||$ and at  regular points
$$\dfrac{g''}{g}=-\sec_M(\dot{c},Z)-
3\frac{|SZ|^2 }{|Z|^2 }\sin^2(\sphericalangle(SZ,Z)).$$
\item[(b)] $g_v(t)$ is continuous for all $t$. Furthermore, it is smooth (and positive) at $t=t^*$ iff $v\bot \ker A_{t^*}$.
\item[(c)] If $\sec_M\ge 0$ (resp. $\sec_M>0$), then $g_v$ is concave
(resp. strictly concave) on any interval where $g_v$ is
positive. If  $g_v$ is constant, then the virtual Jacobi field
$Z$  is a parallel Jacobi field, and if $A_{t_0}=\Id$, then
$Z_t=A_tv$.
\end{enumerate}
  \end{main}
\noindent  Notice that for a surface $g_v=||A_tv||$ is simply the
length of
  the
  Jacobi field.

  \smallskip

  As an immediate consequence one has the following result by
  B.Wilking \cite{Wi} which was crucial in proving the smoothness of the
  Sharafudinov
  projection in the soul theorem: If $M$ has non-negative sectional
  curvature and $A$ is a Lagrange tensor defined along $c$ for {\it all} $t$,
  normalized so that $A_{t_0}=\Id$,
  then one has an orthogonal splitting
  $$
E_{t_0}=\spam\{v\in E_{t_0}\mid A_tv=0 \text{ for some } t\in\R\}\oplus
 \{v\in E_{t_0}\mid A_tv \text{ is parallel for all } t\in\R \}.
  $$

\smallskip

  There is another well known
 concave function in positive curvature given in terms of the volume
 along the geodesic: if $A$ is Lagrange, then  $(\det A_t)^{1/n}$
  is concave if $\Ric\ge 0$.
 One of the advantages of the class of concave functions in Theorem
 A is that by part (b) and (c), some of them are well defined and concave at singular
 points of $A$, whereas $\det A$  vanishes at such points.
  This property of $g_v$ is crucial in our
 applications.

 \smallskip

 There exists a sequence of concave functions interpolating between $g_v$ and
 the volume. For each $p$-dimensional subspace $W\subset E_{t_0}$ set
 $$
g_W(t)=(\det M_t)^{-1/2p} \quad \text{ where }  \ml M_te_i,e_j\mr=\ml (A_t^*)^{-1}e_i,(A_t^*)^{-1}e_j\mr
=\ml\, (A^*A)^{-1}e_i,e_j \mr
 $$
 and $e_1,\dots,e_p$ is an orthonormal basis of $W$. If $W$ is one
 dimensional, $g_W=g_v$ with $v$ a unit vector in $W$, and if $W=E_{t_0}$
 then $g_W=(\det A_t)^{1/n}$.

 Recall that a
manifold is said to have $p$-positive sectional curvature if the sum
of the $p$ smallest eigenvalues of $R(\cdot,v)v$ is positive for all
$v$. Thus $p=1$ is positive sectional curvature and $p=n$ is
positive Ricci curvature.

 \begin{main} Let $A$ be a Lagrange tensor along the geodesic $c$ and $W\subset E_{t_0}$
 a $p$-dimensional subspace.
 \begin{enumerate}
\item[(a)] If the $p$-sectional curvature is non-negative (resp. positive), then $g_W$ is concave
(resp. strictly concave) on any interval where $g_W$ is
positive.
 \item[(b)] $g_W$ is smooth (and positive) at $t=t^*$ iff $W\bot \ker A_{t^*}$.
\item[(c)] If the $p$-sectional curvature is non-negative and $g_W$ is constant, then
$(A_t^*)^{-1}v$ is a parallel Jacobi field for all $v\in W$.
\end{enumerate}
\end{main}

\smallskip

The example of positive curvature in \cite{GVZ} arose from a
systematic study of \emph{cohomogeneity one} manifolds, i.e.,
manifolds with an isometric action whose orbit space is one
dimensional, or equivalently the principal orbits have codimension
one. A classification of positively curved cohomogeneity one
manifolds was carried out  in even dimensions in
 \cite{verdiani:1,verdiani:2}  and   in odd dimensions an exhaustive description was given in \cite{GWZ} of all simply
connected cohomogeneity one manifolds that can possibly support an
invariant metric with positive curvature. In addition to some of the
known examples
 of positive curvature which admit isometric \coo\ actions, two infinite families,
 $P_{k}^7, Q_{k}^7, k\ge 1,$ and one exceptional
manifold $R^7$, all of dimension seven and admitting a \coo\ action
by $\SO(4)$, appeared as the only possible new candidates, see
Section 4 (as well as \cite{Z2})
 for a more detailed description.
  Here $P_1^7$ is the 7-sphere and $Q_1^7$ is the
normal homogeneous positively curved Aloff-Wallach space. The
manifold $P_2^7$ is the new example of positive curvature in
\cite{GVZ}.

These candidates belong to two much larger classes of cohomogeneity
one manifolds depending on 4 integers, described in terms of the
isotropy groups, see Section 4. One is denoted by $P_{
(p_-,q_-),(p_+,q_+)}$, a family of cohomogeneity one manifolds with
$\pi_1=\pi_2=0$, and a second by $Q_{ (p_-,q_-),(p_+,q_+)}$, where
$\pi_1=0,\; \pi_2=\Z$. They all admit a cohomogeneity one action by
$G=\SO(4)$.
  In terms of these, the candidates for positive
 curvature are given by $P_k=P_{(1,1),(1+2k,1-2k)}$,
  $Q_k=Q_{(1,1),(k,k+1)}$, with $k\ge 1$, and the exceptional manifold $R^7=Q_{(3,1),(1,2)}$.

\smallskip

\begin{main}
Let $M$ be one of the $7$-manifolds $Q_{ (p_-,q_-),(p_+,q_+)}$ with
its cohomogeneity one action by $G=\SO(4)$ and assume that $M$ is
not  of type  $Q_k,k\ge 0$. Then there exists no analytic metric
with non-negative sectional curvature invariant under $G$, although
there exists a smooth one.
  \end{main}

  The existence of a smooth metric with non-negative curvature follows from a more general
   result on cohomogeneity one manifolds in \cite{GZ1}. In particular we obtain:

  \begin{cor*}
The exceptional cohomogeneity one manifold $R^7$ does not admit an
invariant metric with positive sectional curvature.
  \end{cor*}

  The method also applies to the family $P_{
(p_-,q_-),(p_+,q_+)}$. Here we will show that if the manifold is not
one of the candidates $P_k$ or of type $P_{ (1,\; q),(p\, ,\, 1)}$,
then there exists a $G$-invariant metric with non-negative sectional
curvature, but no $G$-invariant analytic metric with non-negative
curvature. On the other hand, the exceptional family $P_{ (1,\;
q),(p,1)}$ contains several $G$-invariant analytic metrics with
non-negative curvature since $P_{(1,1),(-3,1)}$ is $\Sph^7$,
$P_{(1,-3),(-3,1)}$ is the positively curved Berger space and
$P_{(1,1),(1,1)}=\Sph^3\times\Sph^4$. We do not know if any of the
other manifolds $P_{ (1,\; q),(p\, ,\, 1)}$ carry analytic metrics
with non-negative curvature.

\bigskip

   The proof of Theorem C is obtained as follows. For a cohomogeneity one $G$-manifold one chooses
    a geodesic $c$ orthogonal to all orbits. Then the action of $G$ induces Killing vector fields on $M$,
    which along $c$
    are Jacobi fields. They  give rise to a Lagrange tensor $A$, to which we can apply Theorem
    A. One then shows that  there exists a Jacobi field $A_tv$,
    and an interval $[a,b]$, such that the corresponding function
    $g_v$ has  derivatives equal to $0$ at the endpoints, and is
    positive on $[a,b]$.
   Thus, if the curvature is non-negative, Theorem A
    implies that $g_v$ is constant on $[a,b]$.
    On the other hand, one shows that $g_v$ must vanish at other singular
  points along $c$ due to smoothness conditions imposed by the group action.
   This implies that there exists a Jacobi field  which is
  parallel on $[a,b]$, but is not parallel at all points along $c$.

\smallskip

We finally discuss an application of Theorem B.  There is a third
family of $7$-dimensional manifolds $N_{p,q}$ on which
$G=\S^3\times\S^3$ acts by cohomogeneity one, see Section 4. We will
show:

\begin{main}
The cohomogeneity one manifolds $N_{p,q}$ have no invariant metric
with $2$-positive sectional curvature, and $N_{1,1}$ has no
invariant metric with $3$-positive sectional curvature.
\end{main}
In contrast, it was shown in \cite{GZ2}  that every simply connected
cohomogeneity one manifold carries an invariant metric with positive
Ricci curvature.

 \bigskip

The differential equation and its applications also hold if we
consider Jacobi fields only in a subbundle invariant under parallel
translation. This arises frequently in the presence of an isometric
group action. For example, a group action  is called polar if  there
exists a so called section $\Sigma$, which is an immersed
submanifold orthogonal to all orbits. Such a section must be totally
geodesic, and hence the group action gives rise to a self adjoint
family of Jacobi fields in the parallel subbundle orthogonal to
$\Sigma$.

   \bigskip

   In Section 1 we recall properties of the Riccati equation and prove Theorem A. In Section 2
   we prove Theorem B  and in Section 3 we
   discuss rigidity properties. Finally, in Section 4, we  prove Theorems C and D.

\smallskip

\bigskip

\smallskip

\section{Concavity}

\bigskip

In this section we present a new concavity result about Jacobi
fields, and first recall some standard notation, see e.g.
\cite{Es3,EH,EO}.
\smallskip

Let $c$ be a geodesic in a Riemannian manifold $M^{n+1}$ defined on
an interval $t_1\le t\le t_2$ and let $E_t=\dot{c}^\perp$ be the
orthogonal complement of $\dot{c}(t)\subset T_{c(t)}M$. For  a
vector field $X$ along $c$, orthogonal to $\dot{c}$, we denote by
$X'$ the covariant derivative $\nabla_{\dot{c}}X$.

Let $V$ be an $n$-dimensional vector space of Jacobi fields along
$c$ orthogonal to $\dot{c}$. Along the geodesic we have that $\ml
X',Y\mr -\ml X,Y'\mr$ is constant for any $ X,Y \in V$. If this
constant is $0$, $V$ is called {\it self adjoint}, i.e.
\begin{equation}\label{selfad}
\ml X',Y\mr =\ml X,Y'\mr , \ \text{ for all } X,Y \in V .
\end{equation}
We call $t$ regular if $X(t),\ X\in V$ span $E$ and singular
otherwise. One easily sees that
\begin{equation}\label{decomp}
E_t=\{X(t) \mid X\in V\}\oplus\{X'(t)\mid
X\in V\ \text{with}\ X(t)=0\}=: V_1(t)\oplus V_2(t)
\end{equation}
for all $t\in [t_1,t_2]$. Notice that self adjointness implies that
the decomposition is orthogonal.  In particular, the singular points
are isolated.

We fix a base point $t_0\in [t_1,t_2]$.  We can then describe the
set of Jacobi fields $V$ by a (smooth) family of linear maps
$A_t\colon E_{t_0}\to E_t$. It is standard to do this by assuming
the base point is regular and define $A_t v=X(t)$ for $X\in V$ with
$X(t_0)=v$. In this case $A_{t_0}=\Id$. But in the applications it
will be useful to allow the base point $t_0$ to be singular as well.

\begin{definition}\label{Jdef}
Let $V$ be selfadjoint family of Jacobi fields and fix $t_0\in
[t_1,t_2]$. Decompose $v\in E_{t_0}$ as $v=v_1+v_2$, $v_i\in
V_i(t_0)$, and define:
\begin{eqnarray*}
&A_t\colon E_{t_0}\to E_t \quad\colon\quad A_t v=X_1(t)+X_2(t)  \\
&\text{ where } X_1,X_2\in V \text{\, with } X_1(t_0)=v_1,\; X_1'(t_0)\in V_1,\; \text{ and }\; X_2(t_0)=0,\; X_2'(t_0)=v_2.
\end{eqnarray*}
\end{definition}

\noindent For this we observe:
\begin{lem}\label{Jinitial} Let $V$ be selfadjoint family of Jacobi fields and choose a base point $t_0$.
\begin{enumerate}
\item[(a)] Given $v\in E_{t_0}$, the Jacobi fields $X_1$ and $X_2$ in
\dref{Jdef} are well defined and unique.
\item[(b)] Given $X\in V$, there exists a unique $v\in E_{t_0}$ such that
$X=A_tv$.
\item[(c)] At
the base point $t_0$ we have, with respect to the orthogonal
decomposition $V_1\oplus V_2$:
$$
A_{t_0}=\left(
  \begin{array}{cc}
    \Id & 0 \\
    0 & 0 \\
  \end{array}
\right)\qquad
A_{t_0}'=\left(
  \begin{array}{cc}
    B & 0 \\
    0 & \Id \\
  \end{array}
\right)
$$
with $B$ self adjoint.\end{enumerate}
\end{lem}
\begin{proof}
(a) Existence of $X_2$ is clear. As for $X_1$,  first choose $Y_1\in
V$ with $Y_1(t_0)=v_1$ and set $Y'_1(t_0)=w_1+w_2$ with $w_i\in
V_i(t_0)$. By \eqref{decomp}, there exists a $Y_2\in V$ such that
$Y_2(t_0)=0$ and $Y_2'(t_0)=w_2$. Then set $X_1=Y_1-Y_2$.
Uniqueness clearly follows from \eqref{decomp} as well.

(b) Given $X\in V$, set $v_1:=X(t_0)$ and $X'(t_0)=w_1+w_2$ with
$w_i\in V_i(t_0)$. There exists a unique $X_2\in V$ with
$X_2(t_0)=0$ and $X_2'(t_0)=w_2$. Setting $X_1:=X-X_2$ we see that
$X=A_tv$ with $v=v_1+w_2$.

Part (c) is clear from the definition and self adjointness.
 \end{proof}

 \noindent Thus
$V$ is indeed uniquely described in terms of $A_t$. Notice though
that $A_{t_0}v=v$ for all $v\in E_{t_0}$ if only  if $t_0$ is
regular.

\smallskip

A point $t$ is regular for $V$ if and only if $A_t$ is invertible.
At regular points $t$ one defines the Riccati operator:
\begin{equation}
S_t\colon E_t\to E_t\ \text{ where }\ S_t v=X'(t) \text{ for  } X\in V \text{ with } X(t)=v,
 \text{ i.e. } A'_t=S_tA_t\; .
\end{equation}
Thus  the family of Jacobi fields $V$ is self adjoint iff $S_t$ is
self adjoint.   $A_t$ satisfies the Jacobi equation and $S_t$ the
Riccati equation:
\begin{equation}\label{Riccati}
A''+RA=0 \ \text{ if and only if }\ S'+S^2+R=0\ \text{ and  }\ A'=SA
\end{equation}
where $R=R_t\colon E_t\to E_t$ is the  self adjoint curvature
endomorphism $R(\ \cdot\ ,\dot{c}),\dot{c}$.

Conversely, let $A_t\colon E_{t_0}\to E_t$ be a solution of
\eqref{Riccati}. We say that $A_t$ is {\it non-degenerate}, if $\ker
A_{t_0}\cap \ker A'_{t_0}=0$. Furthermore, $A_t$ is called a {\it
Lagrange tensor} if $A_t$ is non-degenerate and $S_t$ is self
adjoint. A Lagrange tensor defines an $n$-dimensional family of
Jacobi fields $V=\{A_tv\mid v\in E_{t_0}\}$ which is self adjoint.

  We point out that if $A_t$ is Lagrange,
 then $A_t\circ F$, for any fixed linear isomorphism $F\colon E_{t_0}\to E_{t_0}$, is also a
 Lagrange tensor, in fact with the same tensor $S$.
 Furthermore, if $S_t$ is self adjoint at one point, it is self
 adjoint at all points.
Notice also  that if two Lagrange tensors $A_t$ and $\tilde{A}_t$,
with base points $t_0$ and ${\tilde{t}}_0$, give rise to the same
self adjoint family $V$,  they differ from each other by a linear
isomorphism $F:E_{{t}_0}\to E_{\tilde{t}_0}$. Indeed, if $v\in
E_{t_0}$ and hence $A_tv\in V$, then \lref{Jinitial} implies that
there exists a unique $w\in E_{\tilde{t}_0}$ with $A_t v=\tilde{A}_t
w$. Then $F(v)=w$ clearly defines an isomorphism with
$\tilde{A}_t\circ F=A_t$. This applies in particular if we choose a
different base point  when defining $A_t$ in terms of $V$. Thus
Lagrange tensors, modulo composing with $F$, are in one to one
correspondence with
 n-dimensional vector spaces of Jacobi fields which are self adjoint.

 \smallskip

 From now on let $A$ be a Lagrange tensor. Thus for any  $v\in E_{t_0}$, $A_tv$ is a Jacobi
 field, and
 $t$ is regular if and only if $A_t$ is invertible. Furthermore,
\begin{equation}\label{selfad2}
\ml A'_tv,A_tw\mr=
\ml A_tv,A_t'w\mr \text{ for all } t \text{ and } v,w \in E_{t_0}.
\end{equation}
Notice that here we do not assume that $A_{t_0}$ has any special
form as is the case when $A$ is associated to $V$.
 When clear from context we simply write  $A=A_t,\ S=S_t$.

\smallskip

Let $A_t^*$ be defined by $\ml A^*_tv,w\mr=\ml v,A_tw\mr\ \text{ for
all } v\in E_t,\ w\in E_{t_0}$ and for simplicity set $
(A_t^*)^{-1}={A^{-*}_t}\colon E_{{t}_0}\to E_t$.

\smallskip

\noindent The main purpose of this section is to study the functions
$$g_v(t)=\dfrac{||v||^2}{||{A_t^{-*}} v||}\, , \quad v\in E_{t_0}.$$
The scaling guarantees that $g_{\lambda v}=\lambda g_v$. We first
discuss smoothness properties.

\begin{prop}\label{smooth} Let $A_t$ be a Lagrange tensor and fix a vector $v\in E_{t_0}$. Then
\begin{enumerate}
\item[(a)] The vector field ${A^{-*}_t}v$, and hence the function $g_v$,
 is  smooth outside of the singular set. If $t^*$ is a singular
point, then ${A^{-*}_t}v $ has a smooth extension at $t=t^*$ if
and only if $v$ is orthogonal to $\ker A_{t^*}$.
\item[(b)] $g_v$ is continuous for all $t$ and $g_v(t)> 0$ if and only if $v$ is orthogonal to $\ker A_{t}$.
\end{enumerate}
\end{prop}
\begin{proof} The first claim in part (a) is clear. For simplicity assume that the singular point is $t^*=0$.
Choose an orthonormal basis  $\{e_1,\ldots,e_{n}\}$ of $E_{t_0}$
such that $\{e_{1},\ldots,e_{k}\}$ is a basis of $\ker A_{0}$.

Choose $\e$ such that $A_t$ is non-singular for $t\in (0,\e]$. Then
$A_t$ has a block form (with respect to a parallel basis)
$$A_t=\left(
              \begin{array}{cc}
                tX & Y+tY_2 \\
                tZ & W+tW_2 \\
              \end{array}
            \right) + o(t^2) \text{ and hence }
            {A_t^*}=   \left(
              \begin{array}{cc}
                tX^T & tZ^T \\
                Y^T+tY_2^T & W^T+tW_2^T \\
              \end{array}
            \right) + o(t^2). $$
     We first claim that the matrix
            $$
N=\left(
  \begin{array}{cc}
    X & Y  \\
    Z & W \\
  \end{array}
\right)
            $$
           is  non-singular. This is equivalent to saying that  $ A'e_1,\dots, A'e_k,
           Ae_{k+1},\dots,  Ae_n$ are linearly independent. If not,
           there exists a $v\in \ker A_0$ and $w\in (\ker A_0)^\perp$
           such that $A'v=Aw$. Using self adjointness, $\ml
           Aw,Aw\mr=\ml A'v,Aw\mr=\ml Av,A'w\mr=0$. Thus $Aw=0$ and
           hence  $A'v=0$, which contradicts non-degeneracy.
  In particular, $\det A_t=at^k+ o(t^{k+1})$ with $a$ nonzero.
             It follows that the matrix of minors of $A_t^*$ has the form
$$M=\left(
              \begin{array}{cc}
                t^{k-1}\overline{X} & t^{k-1}\overline{Y} \\
                t^k\overline{Z} &t^k \overline{W} \\
              \end{array}
            \right) +o(t^{k}) \text{ where } \overline{N}=\left(
  \begin{array}{cc}
    \overline{X} & \overline{Y}  \\
    \overline{Z} & \overline{W} \\
  \end{array}
\right)
           $$
 is the matrix of minors of $N^T$, and hence  non-singular. Thus
           $$
            {A_t^{-*}}=\frac{1}{\det A^*} M^T= \frac{1}{\det A}\left(
              \begin{array}{cc}
                t^{k-1}\overline{X}^T & t^{k}\overline{Z}^T \\
                t^{k-1}\overline{Y}^T &t^k \overline{W}^T \\
              \end{array}
            \right)+o(1).
            $$
Hence ${A_t^{-*}}v$ is smooth, and non-zero,  if $v$ is orthogonal
to $\ker A_0$. If $v\in E_{t_0}$ is not orthogonal to $\ker A_0$, we
have $\lim_{t\to 0}||{A_t^{-*}}v||=\infty$ since $\overline{R}$ is
non-singular. Hence $g_v(0)=0$ which finishes (b) as well.
 \end{proof}

 \begin{rem*}  The function $g_v$ provides a lower bound for the norm of the
corresponding Jacobi field, i.e.
$$
g_v\leq ||A_t v||
$$
since
$$\langle v,v\rangle=\langle A^{-1}A v,v\rangle=\langle Av,A^{-*}v\rangle\leq ||Av|| \cdot ||{A}^{-*} v||.$$

 \end{rem*}

\bigskip

\noindent Our main tool is the following differential equation for
$g_v(t)$:
\begin{prop}\label{relation} Let $A$ be a Lagrange tensor and $S=A' A^{-1}$. Then at
regular points we have
\begin{equation}
g_v''+r g_v=0
\end{equation}
where
$$r=  \langle R z , z\rangle + 3  (\; ||S z||^2  -\langle S z,z \rangle^2\; )
 \quad \text{and}\quad z=\frac{ {A_t^{-*}} v}{||{A_t^{-*}}v||  }\; . $$
\end{prop}
\begin{proof}
To simplify the notation we assume $||v||=1$ (which does not effect
the differential equation) and set
$$f_v=\frac 1{g_v^2}=||{A_t^{-*}} v||^2.$$
First observe that
$$
 ({A_t^{-*}})'= -{A_t^{-*}} (A_t^*)' {A_t^{-*}}=-(A_t'A_t^{-1})^* {A_t^{-*}}=-S^*{A_t^{-*}}=-S{A_t^{-*}}
$$
and hence
$$f_v'=-2 \langle S {A_t^{-*}} v,{A_t^{-*}} v\rangle.$$
Furthermore
\begin{eqnarray*}
f_v''&=&-2 \langle S' {A_t^{-*}} v,{A_t^{-*}} v\rangle+4 \langle S^2 {A_t^{-*}} v,{A_t^{-*}} v\rangle=\\
&=&-2 \langle (-S^2-R) {A_t^{-*}} v,{A_t^{-*}} v\rangle+4 \langle S {A_t^{-*}} v,S{A_t^{-*}} v\rangle=\\
&=&2 \langle R {A_t^{-*}} v,{A_t^{-*}} v\rangle+6 ||S{A_t^{-*}} v||^2.
\end{eqnarray*}
and thus
\begin{eqnarray*}
 g_v''&=&
 \left(\frac 34 f_v'^2-\frac 12 f_v'' f_v\right)f^{-5/2}\\
&=&\left(3\langle S {A_t^{-*}} v,{A_t^{-*}} v\rangle^2 -\langle R {A_t^{-*}} v,{A_t^{-*}} v\rangle
||{A_t^{-*}} v||^2
-3 ||S{A_t^{-*}} v||^2 ||{A_t^{-*}} v||^2\right)f^{-5/2}\\
&=&-r g_v.
\end{eqnarray*}
\end{proof}

\begin{rem*}
Notice that ${A_t^{-*}}$ itself satisfies the differential
equation
$$
({A_t^{-*}})''= (2S^2+R){A_t^{-*}}.
$$
\end{rem*}

\noindent \pref{relation} implies certain  concavity properties in
non-negative curvature.

\begin{cor}\label{concave} Let $A$ be a Lagrange tensor. If $R\ge 0$ (resp. $R>0$), then for
any $v\in E_{t_0}$,  $g_v$ is a concave (resp. strictly concave)
function on any interval where $g_v>0$.
\end{cor}

\begin{rem*} The
concavity of $g_v$ implies the convexity of $f_v=||{A_t^{-*}}
v||^2$ (but not conversely). Since $f_v=\ml (A^*A)^{-1}v,v\mr$, this
can also be interpreted as saying the operator $(A^*A)^{-1}$ is
convex. The zeros of $g_v$ correspond to vertical asymptotes of
$f_v$.
\end{rem*}

\begin{example} If $\dim M=2$, then $g_v(t)=||A_tv||$ and hence the concavity  of $g_v$ is indeed a generalization of the concavity of Jacobi
fields in dimension two.
\end{example}

\begin{example}
Let $M=S^3\subset \C^2$ with the standard metric. The restriction of
the action field of the Hopf action of $S^1$ to a geodesic is a
Jacobi field $J_1$ with unit length. Consider the geodesic
$c(t)=(\cos(t),\sin(t))$, then $J_1=i\,c(t)=(i\cos(t),i\sin(t))$.
Let $J_2=(0,i\sin(t))$ then $span\{J_1,J_2\}$ is a self-adjoint
family of Jacobi fields  $V$ along $c(t)$. The singular points along
$c(t)$ are $t=n \frac \pi 2, \, n\in \Z$, since $J_2=0$ for $t=n
\pi$ and $J_1-J_2=0$ for $t=(2n+1) \frac \pi 2$. Now $t_0=\frac \pi
4$ is a regular point and, if $v=J_1(t_0)$, one easily sees that
$g_v(t)=|\sin(2t)|\leq ||J_1(t)||=1$. Notice also that $g_{w}$ with
$w=J_2(t_0)$ is smooth across the singularity at $\frac{\pi}{2}$.
\end{example}

\begin{example} If $\sec_M\ge \delta$, then $g_v''+rg_v=0$ with $r\ge \delta$.
Thus Sturm comparison implies that $g_v\le f_\delta$ with
$f_\delta''+\delta f=0$ and $f_\delta(t_0)=g_v(t_0)=|J|(t_0)$,
$f_\delta'(t_0)=g_v'(t_0)=|J|'(t_0)$ (see \pref{higher} below). This
comparison holds up to the first point where $g_v$ vanishes.

In contrast, the usual Rauch comparison theorem implies that $|J|\le
f_\delta$, but only holds up to the first singularity of $A_t$, i.e.
there could be other Jacobi fields $A_tw$ which vanish before $|J|$.

For $g_v$ one obtains an upper bound on $[t_0,t_1]$ as long  $v$ is
orthogonal to the kernels of $A_t\, , t\in[t_0,t_1]$, or
equivalently $g_v>0$. Of course a zero of $g_v$ also corresponds to
a singularity of $A_t$. For example, if $\sec_M\ge 1$, this implies
that the index of the geodesic is at least $n-1$ after length $\pi$.

We remark that for an upper curvature bound $\sec_M\le \mu$, one can
analogously use the differential equation for $|J|$ in the
Introduction to get the usual lower bound on $|J|$, without having
to prove a Rauch comparison theorem.
\end{example}
\smallskip

It is useful to compare the higher derivatives of $g_v$ with those
of $||A_tv||$.

\begin{prop}\label{higher} Let $A$ be the Lagrange tensor
defined by a self adjoint family of Jacobi fields $V$ as in
\eqref{Jdef} with base point $t_0$ and $v\perp \ker A_{t_0}$. Then
 for $t=t_0$ we have:
$$g_v=||Av||\quad ,\quad g_v'=||Av||' \quad ,\quad g_v''=||Av||''-4
\left(||A'v||^2- \ml A'v,v\mr^2\right)\le ||Av||''. $$
  \end{prop}
  \begin{proof}
The assumption  $v\perp \ker A_{t_0}$ implies that $g_v$ is smooth
and non-zero at $t_0$. But to determine its value and derivatives at
$t=t_0$ we need to carefully take the limit as $t\to t_0$.

Since the equations are scale invariant, we can assume $||v||=1$.
Recall that at the base point we have ${A_{t_0}}_{| V_1}=\Id$,
${A_{t_0}}_{| V_2}=0$ and $V_1\perp V_2$.
 Thus $v\in V_1$ and hence
  $A_{t_0}v=v$, as well as $A_{t_0}^*v=v$.

  To compute the derivatives of $g$,  recall that \pref{smooth} also implies that
  $A_t^{-*} v$ is smooth at $t=t_0$. We begin by showing that:
  \begin{equation}\label{limit}
  \lim_{t\to t_0}A_{t}^{-*} v=v,\quad  \lim_{t\to t_0}(A_t^{-*} v)'=-A_{t_0}'v\ .
  \end{equation}
  For the first claim, observe that \lref{Jinitial} implies that $A_{t_0}'w=w$ if $w\in V_2$. Thus, in the language of the proof of \pref{smooth}, it
  follows that $N=\Id$ and hence $\overline{N}=\Id$ as well. Furthermore, $\det A_t=t^k+ o(t^{k+1})$ and thus the formula for the inverse implies
  that
   $\lim_{t\to t_0}A_t^{-*} v=v$.

  For the second claim we first observe that $ \lim_{t\to t_0}(A_t^{-*} v)'\in V_1
  $ since for $w\in V_2$ we have that $A_t'w$ and $w$ have the
  same limit and
  $$
\ml w, \lim_{t\to t_0} A_t^{-*} v\mr= \lim_{t\to t_0}\ml A'_tw,A_t^{-*} v\mr=-\lim_{t\to t_0}\ml A_{t}w,  (A_t^{-*} v)'\mr=
-\ml A_{t_0}w, \lim_{t\to t_0} (A_t^{-*} v)'\mr=0
  $$
  where the second equality follows  by differentiating $
\langle w,v\rangle= \langle  A_t^{-1} A_tw, v\rangle=
 \langle  A_tw, A_t^{-*} v\rangle
$. Now, if $w\in V_1$ we have
$$
 \lim_{t\to t_0} \langle (A_t^{-*}v)' ,A_tw\rangle=- \lim_{t\to t_0} \langle A_t^{-*}v ,A_t'w\rangle
 =-\lim_{t\to t_0} \langle A_tv ,A_t'w\rangle=-\lim_{t\to t_0} \langle A_t'v ,A_t w\rangle
$$
where we have used the fact that $ A_t^{-*}v$ and $A_tv$ have the
  same limit.
This implies the second part of \eqref{limit} since  $A_{t_0}w=w$.

We now apply \eqref{limit} to $g$. First, note that
  $g_v(t_0)=1=||A_{t_0}v||$.
 For the derivative, using $f_v(t)=||A_t^{-*} v||^2$, we see
 that
$$
 g_v'(t_0)=-\frac 12 \lim_{t\to t_0}
  \frac {f_v'(t)}{f_v(t)^{\frac 32}}=
-\frac{\lim_{t\to t_0} \langle (A_t^{-*} v)',A_t^{-*} v\rangle}
{ \lim_{t\to t_0}||{A_t^{-*}}v||^3}=
-\lim_{t\to t_0} \langle (A_t^{-*} v)',A_t^{-*} v\rangle
$$
and thus
$$
g_v'(t_0)=\ml A'v,v\mr=\ml A'v,Av\mr=||Av||'_{t_0}.
$$

\noindent For the second derivative, we use the differential
equation from \pref{relation} for $g_v$:
$$g_v''(t_0)=-\lim_{t\to t_0} r g_v=-\lim_{t\to t_0} \left\{3 (||Sz||^2- \ml Sz,z\mr^2)-\ml R z,z\mr\right\}$$
where $z= A_t^{-*} v/||A_t^{-*} v|| $. From the proof of
\pref{relation}, recall that at regular points we have $SA_t^{-*}
v=-(A_t^{-*} v)'$ and hence \eqref{limit} implies that $\lim_{t\to
t_0} Sz=A'_{t_0}v$. Thus

$$g_v''(t_0)=-3 (||A'v||^2- \ml A'v,v\mr^2)-\ml R v,v\mr$$
and since
$$
||Av||''=\frac{-\langle R Av,Av\rangle ||Av||^2+||A'v||^2 ||Av||^2- \langle A'v,Av\rangle^2 }{||Av||^3}
$$ we have
$$g_v''(t_0)=||Av||''-4 (||A'v||^2- \ml A'v,v\mr^2).$$
\end{proof}

\smallskip

\bigskip

\bigskip

\section{Concavity of Volumes}

\bigskip

We  construct a collection of concave functions which contain $g_v$
as a special case. For this we fix a $p$-dimensional subspace
$W\subset E_{t_0}$ and choose an orthonormal basis $e_1,\dots,e_p$
of $W$. Define
\begin{equation}\label{Mmatrix}
M\colon W\to W\quad \text{ with }\quad \ml M_te_i,e_j\mr=\ml A_t^{-*}e_i,A_t^{-*}e_j\mr
=\ml (A^*A)^{-1}e_i,e_j\mr,\quad 1\le i,j\le p\; .
\end{equation}

Thus $M$ represents the upper $p\times p$ block of the matrix $
(A^*A)^{-1}$. Furthermore, we decompose  $S=A'A^{-1}$, where we have
set $W_t:=A_t^{-*}W $, as
$$
S_1\colon W_t\to W_t,\quad S_2\colon W_t\to W_t^\perp\quad \text{ with }\quad
Sw=S_1w+S_2w \ \text{ for all } \ w\in W_t.
$$
 Notice that $S_1$ is again a symmetric endomorphism. Notice also
 that since $(A^*A)^{-1}$ is positive definite at regular points, so is the upper
 $p\times p$ block by Sylvester's theorem and thus $\det M_t> 0$.

\begin{prop}\label{volume}
Let $A$ be a Lagrange tensor and $W\subset E_{t_0}$  a
$p$-dimensional subspace. Then at  regular points  the function
$$
g_W(t)= (\det M_t)^{-1/2p}
$$
 satisfies the differential equation
$$
p\, \frac{g''}{g}=\frac1p(\tr S_1)^2-\tr(S_1^2)-3\tr(S_2^TS_2)-\sum_{i=1}^{i=p}  \ml \, Rw_i,w_i\mr
$$
where $w_i$ is an orthonormal basis of $W_t$.
\end{prop}
\begin{proof}

 As in the proof of \pref{relation}, one easily sees
that
\begin{equation}\label{Mderivative}
 \ml M'e_i,e_j\mr=-2\,\ml SA^{-*}e_i,A^{-*}e_j\mr
\quad , \quad \ml M''e_i,e_j\mr=\ml \, (6S^2+2R)A^{-*}e_i,A^{-*}e_j\mr.
\end{equation}
For convenience, set $f=\det M_t$. Differentiating we obtain:
\begin{equation}
f'=(\det M)'=\det M\tr(M^{-1}M'),\quad \text{ or }\quad \frac{f'}{f}=\tr(M^{-1}M')
\end{equation}
and hence
\begin{eqnarray*} \frac{f''}{f}&=&\,\left[\tr(M^{-1}\,M')\right]^2+\,
\tr((M^{-1})'\,M')+\tr (M^{-1}\, M'')\\
&=&\left[\tr(M^{-1}\,M')\right]^2+\,
\tr(-M^{-1}\,M'\,M^{-1}\,M')+\tr (M^{-1}\, M'')\\
&=&\left[\tr(M^{-1}\,M')\right]^2\,
-\tr(\left[M^{-1}\,M'\right]^2)+\tr (M^{-1}\, M'').
\end{eqnarray*}

We now examine each term separately. For this, fix a regular point
$t^*$ and choose an orthonormal basis $e_1,\dots,e_p$ of $W$ which
diagonalizes the symmetric matrix $M_{t^*}$, i.e. $\ml
A^{-*}_{t^*}e_i,A^{-*}_{t^*}e_j\mr=||
A^{-*}_{t^*}e_i||^2\,\delta_{i,j}. $ Thus $
Z_i:=\frac{A_{t^*}^{-*}e_i}{||A_{t^*}^{-*}e_i||}$ is an orthonormal
basis of $W_{t^*}$.

Dropping the index $t^*$ from now on, the entries of $M^{-1}M'$ are
$\frac{-2}{||A^{-*}e_i||^2}\ml SA^{-*}e_i,A^{-*}e_j\mr$ and thus
$$
\tr(M^{-1}\,M')=-2\sum_{i=1}^{i=p}\frac{1}{||A^{-*}e_i||^2}\ml SA^{-*}e_i,A^{-*}e_i\mr=
-2\sum_{i=1}^{i=p}\ml SZ_i,Z_i\mr=-2\tr S_1.
$$
For a general matrix $B=(b_{ij})$ we have $\tr
B^2=\sum_{i,j}b_{ij}b_{ji}$ and hence
\begin{eqnarray*}
\tr(\left[M^{-1}\,M'\right]^2)&=&4\sum_{i,j}
\frac{ \ml SA^{-*}e_i,A^{-*}e_j\mr\ml SA^{-*}e_j,A^{-*}e_i\mr }{ ||A^{-*}e_i||^2||A^{-*}e_j||^2 }\\
&=&4\sum_{i,j}\ml SZ_i,Z_j\mr^2=4\sum_{i,j}\ml S_1Z_i,Z_j\mr^2=4\tr (S_1^2).
\end{eqnarray*}
 Finally

\begin{eqnarray*}  \tr(M^{-1}M'')&=&  \sum_i \frac{1}{||A^{-*}e_i||^2}\ml \,
(6S^2+2R)A^{-*}e_i,A^{-*}e_i\mr\\
&=&6\sum_i \ml \, S^2Z_i,Z_i\mr+2\sum_i  \ml \, RZ_i,Z_i\mr
=6\sum_i \ml \, SZ_i,SZ_i\mr+2\sum_i  \ml \, RZ_i,Z_i\mr\\
&=&6\sum_i \ml \, S_1Z_i,S_1Z_i\mr+6\sum_i \ml \, S_2Z_i,S_2Z_i\mr+2\sum_i  \ml \, RZ_i,Z_i\mr\\
&=& 6\tr (S_1^2) +6 \tr (S_2^TS_2)+2\sum_i  \ml \, RZ_i,Z_i\mr.
\end{eqnarray*}
Altogether
$$
\frac{f'}{f}=-2\tr S_1\quad ,\quad \frac{f''}{f}=4(\tr S_1)^2+2\tr(S_1^2)+6\tr(S_2^TS_2)+2\sum_i  \ml \, RZ_i,Z_i\mr.
$$
For the function $g=f^{-1/2p}$ we have
$$
2p\, \frac{g''}{g}=-\frac{f''}{f}+\frac{2p+1}{2p}\left(\frac{f'}{f}\right)^2
=\frac{2}{p}(\tr S_1)^2-2\tr(S_1)^2-6\tr(S_2^TS_2)-2\sum_i  \ml \, RZ_i,Z_i\mr
$$
which proves our claim. \end{proof}

\begin{rem*}
If $W$ is one dimensional, clearly $g_W=g_v$ for $v$ a unit vector
in $W$. If $W=E_{t_0}$, we have $\det M=\det (A^*A)^{-1}=1/(\det
A)^2$ and thus $g_W=(\det A)^{1/n}$. The differential equation in
this case reduces to $\,n\, g''/g=\frac1n(\tr
S)^2-\tr(S^2)-\Ric(\dot{c},\dot{c})$ giving rise to the well known
concavity of the volume in positive Ricci curvature. Notice also
that the concavity of $g_W$ already holds under the assumption that
the curvature is $p\; $-positive, i.e. the sum of the  $p$ smallest
eigenvalues of $R$ are positive.
\end{rem*}

{\it Proof of Theorem B}\,:\, We first prove part (b). If $W\perp
\ker A_{t^*}$, then \pref{smooth} implies that $A^{-*}v$ is smooth
at $t^*$ for any $v\in W$, and hence $M_{t}$ is smooth at $t^*$ as
well. The proof of \pref{smooth} also shows that if $e_1,\dots,e_p$
is a basis of $W$, then $A^{-*}e_1,\dots,A^{-*}e_p$ are linearly
independent at $t=t^*$ and hence $g_W(t^*)>0$. It also follows that
if $W$ is not orthogonal to $\ker A_{t^*}$, then $g_W(t^*)=0$.

To prove part (a), first recall that $(x_1+\dots +x_p)^2\le
p(x_1^2+\dots+x_p^2)$ with equality if and only if all $x_i$ are
equal to each other. Thus $(\tr S_1)^2-p\,\tr(S_1^2)\le 0$ with
equality iff $S_1=\lambda\Id$. Furthermore, if the sum of the $p$
smallest eigenvalues of $R$ are non-negative, one easily sees that
$\sum_{i=1}^{i=p} \ml \, Rw_i,w_i\mr\ge 0$ if $w_1,\dots,w_p$ is an
orthonormal basis of any $p$ dimensional subspace of $E_{t_o}$.
Finally, $S_2^TS_2$ is clearly positive semi-definite. Altogether,
\pref{volume} implies that $g_W$ is concave.

If $g$ is constant, the differential equation implies that for any
$v\in W_t$ we have $S_1v=\lambda v$ for some function $\lambda(t)$.
Furthermore, $0=\ml S_2^TS_2v,v\mr=\ml S_2v,S_2v\mr$ and hence
$S_2v=0$. In other words, $Sv=\lambda v$ for all $v\in W_t$. But if
$g$ is constant $f$ is constant as well and $f'=0$ implies that $\tr
S_1=0$ and hence $\lambda=0$. Thus $(A^{-*}v)'=-SA^{-*}v=0$, for all
$v\in W$, which implies that the function $g_{v}$  is constant, and
hence by \pref{rigidity} below, $A^{-*}v$ is a parallel Jacobi
field. This proves part (c). \qed

\bigskip

\bigskip

\section{Rigidity}

\bigskip

We now  use the results in Section 1  to prove  the existence of
parallel Jacobi fields in non-negative curvature, i.e. vectors $v\in
E_{t_0}$ with $A'_tv=0$.
 We allow endpoints and interior points of the geodesic to be singular.

\begin{prop}\label{rigidity} Let $A$ be a Lagrange tensor along
 the geodesic $c\colon[t_0,t_1]\to M$. If $R\ge 0$ and if there
exists a non-zero vector $v\in E_{t_0}$ such that
\begin{enumerate}
\item[(a)] $g_v'(t_0)=g_v'(t_1)=0$,
\item[(b)] $v$ is orthogonal to $\ker A_t$ for all $t_0\le t\le t_1$,
\end{enumerate}
then  $w=A^{-1}_t{A_t^{-*}}v\in E_{t_0}$ is constant and $A'_tw=0$
for all $t$. Thus ${A_t^{-*}}v=A_tw$ is a parallel Jacobi field.
\end{prop}
\begin{proof}
By \pref{smooth}, assumption (b) implies that $g_v(t)$
 is smooth and positive for all $t_0\le t\le t_1$,  and by \cref{concave}, $g_v$ is concave and hence  constant.
 Thus $f_v=||{A_t^{-*}}v||$ is constant as well. At regular points
 we thus have
$$0=f_v''=2 \langle R {A}^{-*} v,{A}^{-*} v\rangle+6 ||S{A}^{-*} v||^2$$
and hence $S{A}^{-*} v=0$. Thus $({A}^{-*}
v)'=-S{A^{-*}}v=0$ and hence
$$(A^{-1}{A}^{-*} v)'=-A^{-1}A'A^{-1}{A}^{-*} v=-A^{-1}S{A}^{-*} v=0.$$
  Therefore, on any connected component of the regular points $A^{-1}{A}^{-*} v=w$ is constant and $A w={A}^{-*} v$ is
parallel. Since ${A}^{-*}v$ is continuous, $Aw$ is parallel for
all $t$.
\end{proof}

Here is one possibility to translate \pref{rigidity} into a
statement about Jacobi fields only, which is what we will use for
the obstruction in Section 4.

\begin{prop}
\label{rigidity2} Let $M^{n+1}$ be a manifold with non-negative
sectional curvature and $V$ a self adjoint family of Jacobi fields
along the geodesic $c : [t_0 , t_1 ] \to M$. Assume there exists
$X\in V$ such that
\begin{enumerate}
 \item[(a)] $||X||_{t}\neq 0$, $||X||'_{t}=0$ for $t=t_0$ and $t=t_1$,
 \item[(b)] If $Y\in V$ and $\ml X(t_1),Y(t_1)\mr=0$ then $\langle X(t_0),Y(t_0)\rangle=0$,
 \item[(c)] If $Y\in V$ and $Y(t)=0$ for some $t\in (t_0,t_1)$ then $\langle X(t_0),Y(t_0)\rangle
 =0$,
 \item[(d)] If $Y(t_0)=0$, then $\langle X'(t_0),Y'(t_0)\rangle=0$,
\end{enumerate}
Then $X$ is a parallel Jacobi field along $c$.
\end{prop}

\begin{proof}
We choose as a base point $t=t_0$. Then $V$ defines Lagrange tensor
$A_t$ as in \eqref{Jdef} with  ${A_{t_0}}_{| V_1}=\Id,\ {A_{t_0}}_{|
V_2}=0$ and $V_1\perp V_2$. By $(a)$ we have that $X(t_0)\neq 0$ and
we set $v:=X(t_0)\in V_1$. If $Y\in V$ and $Y(t_0)=0$ then
$Y'(t_0)\in V_2$ and $V_2$ is spanned by such vectors. Thus  $(d)$
implies $X'(t_0)\in V_1$ and hence by the definition \eqref{Jdef} we
have $X(t)=A_tv$, and $A_{t_0}v=v$.

We now want to show that the assumptions of \pref{rigidity} are
satisfied by $A_t$. We start with the second part.

Let $w\in \ker(A_t)$, i.e. $A_tw=0$ for $t\in (t_0,t_1)$. Set
$w=w_1+w_2$ with $w_i\in V_i(t_0)$ and hence $A_{t_0}w=w_1$.
Assumption (c) implies that $\ml A_{t_0}v,A_{t_0}w\mr=\langle
v,w_1\rangle=0$. Since $\langle v,V_2\rangle=0$ as well, we have
$\langle v,w\rangle =0$ and hence  $v\perp \ker A_t$. The same
argument shows that $v\perp \ker A_{t_1}$ by using  (b). If
$A_{t_0}w=0$, then $w\in V_2$ and hence $\ml v,w\mr=0$. Thus
\pref{smooth} implies that $ {A_t}^{-*}v$ and hence $g_v$ is smooth
for all $t\in[t_0,t_1]$.

We now show that $g_v'$ vanishes at the endpoints.  By
 \pref{higher}, $g_v'(t_0)= ||Av||'= ||X||'(t_0)=0  $.
 For $t=t_1$ the proof is similar to the proof of \pref{higher}. We first claim that
  \begin{equation}\label{limit2}
  \lim_{t\to t_1}A_t^{-*}v=\lambda A_{t_1} v \ \text{ for some } \lambda\in\R.
  \end{equation}

  To see this, we begin by showing that $ \lim_{t\to
  t_1}A_t^{-*}v\in V_1(t_1)$. But $V_1(t_1)\perp V_2(t_1)$
  and $V_2(t_1)$ is spanned by $A_{t_1}'w$ for some $w\in E_{t_0}$ with
  $A_{t_1}w=0$. By differentiating $\ml w,v\mr= \ml  A_t^{-*}v,A_tw   \mr$ we obtain
  $$\ml \lim_{t\to t_1} A_t^{-*}v, A_{t_1}'w \mr=\lim_{t\to t_1}\ml  A_t^{-*}v, A_t'w \mr=- \lim_{t\to t_1}\ml
  (A_t^{-*})'v,
  A_{t}w\mr
  =-\ml \lim_{t\to t_1} (A_t^{-*})'v , A_{t_1}w\mr
  =0.$$
   Next, we show that  $\ml \lim_{t\to
  t_1}A_t^{-*}v, A_{t_1}w\mr=0$ whenever $\ml
  A_{t_1}v,A_{t_1}w\mr=0$, which clearly implies \eqref{limit2} since $\Im A_{t_1}=V_1(t_1)$. To
  see this, we observe that (b)  implies  $0=\ml A_{t_0}w,A_{t_0}v\mr
 =\ml w_1,v\mr=\ml w_1+w_2,v\mr=\ml
  w,v\mr$ and hence
  $$\ml \lim_{t\to t_1} A_t^{-*}v, A_{t_1}w \mr= \lim_{t\to t_1}\ml
  A_t^{-*}v,
  A_{t}w\mr
  =\ml v,w\mr=0.$$

\smallskip

We now use \eqref{limit2} to show that $g_v'(t_1)=0$. Since
$g_v(t_1)\ne 0$ by (a), this is equivalent to $f_v'(t_1)=0$. By
\eqref{limit2}, $A_t^{-*}v$ and $\lambda A_{t} v $ have the same
limit and thus
\begin{align*}
 f_v'(t_1)=&
2\lim_{t\to t_1} \langle (A_t^{-*}v)',A_t^{-*}v\rangle
=
2\lim_{t\to t_1} \langle (A_t^{-*}v)',\lambda A_{t} v \rangle\\
&=
-2\lambda\lim_{t\to t_1}\langle A_t^{-*}v, A_t'v,\rangle
=-2\lambda^2\langle A_{t_1}v,A'_{t_1} v\rangle=  -\lambda^2 (||Av||^2)'_{t=t_1}.
\end{align*}
which is $0$ since $||Av||'(t_1)=||X||'(t_1)=0$.

 \pref{rigidity} now implies that ${A_t}^{-*}v=Aw$, for some $w\in E_{t_0}$, is a parallel Jacobi
field in $V$ and   ${A^{-*}_{t_0}}v=v=A_{t_0}w$ by \eqref{limit}.
Since $A'_{t_0}w=0$, \eqref{Jdef} implies that  $A_{t_0}w=w$, and
hence $w=v$ and thus $A_tw=A_tv=X$ is a parallel Jacobi field.
\end{proof}

\begin{rem*} (a)
Notice that the first three   conditions are necessary for $X$ to be
parallel, using, for (b) and (c) that in a self adjoint family of
 Jacobi fields, $\ml X,Y\mr'=\ml X ,Y'\mr=\ml X' ,Y \mr=0 $ for all $X,Y\in V$ with $X$ parallel.
 If there are no interior singular points, (b) is the only global condition and relates the Jacobi
   fields at $t_0$ and $t_1$.
Some global condition is clearly necessary since there are Jacobi
fields of constant length (restricted to a geodesic with no
singularities) which are not parallel.

Also notice that assumption (d) is necessary since on
$M=\Sph^1\times\Sph^2$ with the product metric we can take the
geodesic $c(t)=(1,\gamma(t))$ with $\gamma$ a great circle from
north pole to south pole. Then $V=\spam\{Z_1,Z_2 \}$ with
$Z_1=(1,0),\; Z_2=(0,Y(t))$ and $Y$ a Jacobi field vanishing at
north and south pole is a self adjoint family along $c$. Setting
$X=Z_1+Z_2$ one sees that all conditions in \pref{rigidity2}, except
for (d), are satisfied, but $X$ is not parallel.

\smallskip

(b) The fact that assumption (d) makes the Proposition asymmetric is
due to the fact that the definition of $g_v$ involves the choice of
a base point. This  turns out to be quite useful since for the
manifolds in Section 3, (d) is sometimes satisfied at one endpoint,
but not necessarily at the other. Of course, if $t_0$ is regular,
condition (d) is empty.
\end{rem*}

\begin{prop}\label{nonparallel}
Let $V$ and $X\in V$ satisfy the conditions in \pref{rigidity2} and
assume that $V$ is defined on a larger interval $[t_0,t_2]\supset
[t_0,t_1]$. If there exists a Jacobi field $Y\in V$ such that
$Y(t^*)=0$ for some $t^*\in (t_1,t_2]$ and $\langle
X(t_0),Y(t_0)\rangle\ne 0$, then $X$ is not parallel on $[t_0,t_2]$.
\end{prop}
\begin{proof}
Let $A_t$ be the Lagrange tensor associated to $V$ with base point
$t_0$. Recall that in the proof of \pref{rigidity2} we showed that
$X(t)=A_tv$ with $v=X(t_0)$. The assumption that $\langle
X(t_0),Y(t_0)\rangle\ne 0$ means that
 $v$ is not orthogonal to $\ker
A_{t^*}$ and hence $g_v(t^*)= 0$ by \pref{smooth} (b). Now assume
that $X$ is parallel on $[t_0,t_2]$. We claim that in that case
$g_v(t)$ would be constant on $[t_0,t_2]$, contradicting that fact
that $g_v(t^*)= 0$.

To see this, we show that $A'_tv=0$ with $A_{t_0}v=v$ implies
$g_v(t)=||A_tv||$. First observe that by self adjointness $\ml
A_tv,A_tw\mr'=\ml A'_tv,A_tw\mr+\ml A_tv,A'_tw\mr= 2\ml
A'_tv,A_tw\mr=0$. Thus if $\ml v,w\mr=0$, we have $\ml
A_tv,A_tw\mr=\ml A_{t_0}v,A_{t_0}w\mr=\ml v,w_1\mr=\ml v,w\mr=0$.
Furthermore, at regular points  $\ml v,w\mr=\ml A_t^{-*}v,A_tw\mr$
and hence $ A_t^{-*}v=\lambda A_tv$ for some function $\lambda$. But
then $\ml v,v\mr=\ml A_t^{-*}v,A_tv\mr=\lambda \ml
A_tv,A_tv\mr=\lambda\ml v,v\mr $ and thus $\lambda=1$, i.e.
$A_t^{-*}v=A_tv$ for all regular $t$. Thus  $g_v(t)=||v||^2/||
A_t^{-*}v||=||v||^2/|| A_tv||=||v||=||A_tv||$ for all regular $t$
and hence for all $t$.
\end{proof}

\bigskip

\section{Proof of  Theorem C and D}

\bigskip

We now use \pref{rigidity2} and \pref{nonparallel}  to prove Theorem
C and D.

\bigskip

A simply connected compact cohomogeneity one manifold is  the union of two
homogeneous disc bundles. Given  compact Lie groups $H,\, \Km ,\,
\Kp$ and $\G$ with inclusions $H\subset \Kpm \subset G$ satisfying
$\Kpm/H=\Sph^{\ell_\pm}$, the transitive action of $\Kpm$ on
  $\Sph^{\ell_\pm}$ extends to a linear action on the disc $\Disc^{{\ell_\pm}+1} $.
We can thus define
$M=G\times_{\Km}\Disc^{{\ell_-}+1}\cup G\times_{\Kp}\Disc^{{\ell_+}+1}$
glued along the boundary $ \partial (G\times_{\Kpm}\Disc^{\ell_\pm+1})=G\times_{\Kpm}\Kpm/H=G/H$
via the identity. $G$  acts on $M$ on each half via left action in the first component. This action has principal isotropy group $H$ and singular isotropy groups $\Kpm$.
 One possible description of a cohomogeneity one manifold is thus
simply in terms of the Lie groups $H\subset \{\Km , \Kp\}\subset G$
(see e.g. \cite{AA}).

\smallskip

The first family of cohomogeneity one manifolds we denote by  $P_{
(p_-,q_-),(p_+,q_+)}$
 and is given by the group diagram
$$
 H=\{\pm (1,1),\pm
(i,i),\pm (j,j),\pm (k,k)\} \subset\{ (e^{ip_-t},e^{iq_-t})\cdot H\;
   ,\;
   (e^{jp_+t},e^{jq_+t})\cdot H
   \}\subset\S^3\times\S^3.
  $$
  where $\gcd(p_-,q_-)=\gcd(p_+,q_+)=1$ and all $4$ integers are
   congruent to $1$ mod $4$.

  The second family  $Q_{
(p_-,q_-),(p_+,q_+)}$
  is given by the group diagram
$$
 H=  \{(\pm
1, \pm 1) , (\pm i , \pm i)\} \subset\{ (e^{ip_-t},e^{iq_-t})\cdot H\;
   ,\;
   (e^{jp_+t},e^{jq_+t})\cdot H
   \}\subset\S^3\times\S^3,
  $$
   where $\gcd(p_-,q_-)=\gcd(p_+,q_+)=1$, $q_+$ is even, and
   $p_-,q_-,p_+$ are
   congruent to $1$ mod $4$.

  \smallskip

  The candidates for positive curvature in \cite{GWZ} are $P_k=P_{(1,1),(1+2k,1-2k)}$,
  $Q_k=Q_{(1,1),(k,k+1)}$ with $k\ge 1$, and the exceptional manifold $R^7=Q_{(-3,1),(1,2)}$.

\bigskip

We now describe the  geometry of a general \coo\ action.
A $G$ invariant metric  is determined by its restriction
to  a geodesic $c$  normal to all orbits.
 At the  points $c(t)$ which are regular with respect to the action of $G$,
 the isotropy is constant  and we denote it by  $H$.
 In terms of a fixed biinvariant inner product $Q$ on the Lie algebra $\fg$
 and corresponding $Q$-orthogonal splitting $\fg=\fh\oplus\fh^\perp$ we
 identify, at regular points,   $\dot{c}^\perp\subset T_{c(t)}M$  with $\fh^\perp$
 via action fields: $X\in\fh^\perp\to
X^*(c(t))$. $H$ acts on $\fh^\perp$ via the adjoint representation
and a $G$ invariant metric on $G/H$ is described by an $\Ad(H)$
invariant inner product on $\fh^\perp$. Along $c$ the metric on $M$
is thus described by a collection of functions, which at the
endpoint must satisfy certain smoothness conditions.

\smallskip

 Since $G$ acts by isometries, $X^*,\ X\in\fg$, are Killing vector
 fields
  and hence the restriction to a geodesic is a Jacobi field.
  This gives rise to an $(n-1)$-dimensional  family of Jacobi
  fields along $c$ defined by $V:=\{X^*(c(t))\mid X\in\fh^\perp\}$.
  The self adjoint shape operator $S_t$ of the regular hypersurface orbit $G/H$ at $c(t)$ satisfies
$\nabla_{\dot c(t)}X^*=\nabla_{X^*}\dot c=S_t(X^*(c(t)))$, i.e. $X'=S_t(X),\ X\in\fh^\perp $.
Hence $V$ is self adjoint.
 \smallskip

 A singular point of $V$ is a point $c(t_0)$ such that there exists an $X^*\in V$ with $X^*(c(t_0))=0$,
  i.e. the isotropy group $G_{c(t_0)}$ satisfies $\dim G_{c(t_0)} > \dim H$ and is thus a singular
   isotropy group of the action. For simplicity set $K:=G_{c(t_0)}$
 and  define a  $Q$-orthogonal decompositions
 $$\fg=\fk\oplus \fm  , \quad  \fk=\fh\oplus \fp \ \text{ and thus }\ \fh^\perp=\fp \oplus \fm . $$
  Here $\fm$ can be viewed as the tangent space to the singular orbit $G/K$ at
  $c(t_0)$. The slice $D$, i.e. the vector space normal to $G/K$ at
  $c(t_0)$, can be identified with  $D :=\dot c(t_0)\oplus \fp$
  where $\fp\subset D$ via $X\in \fp\to (X^*)'(c(t_0))$. Notice that
  $X^*(c(t_0))=0$.
Since the slice is orthogonal to the orbit, we have   $ \ml  (X^*)'
, Y^*\mr_{c(t_0)}=0$ for $X\in \fp$ and $Y\in \fm$. $K$ acts via the
isotropy action $\Ad(K)_{| \fm}$ of $G/K$ on $\fm$ and via the slice
representation on $D$.
  The second fundamental form of the singular orbit can be viewed as a linear map
   ${B}\colon D\to S^2(\fm)$, $N\to \{ (X,Y)\to \ml S_N(X),Y\mr\}$. Since $K$ acts by isometries, ${B}$ is
 equivariant with respect to the slice representation of $K$ on $D$ and the action on $S^2(\fm)$ induced by its  isotropy representation
on $\fm$. An $\Ad(K)$ invariant irreducible splitting
$\fm=\fm_1\oplus\dots\oplus\fm_r$ induces a splitting of
 $S^2(\fm)$ into irreducible summands.  If for some $i$, the slice representation (which is irreducible)
  is not a subrepresentation of
   $S^2(\fm_i)$, this implies that $\ml {S_{\dot c(t_0)}}X,Y\mr=\ml X',Y\mr_{c(t_0)}=0$ for $X,Y\in
   \fm_i$. In particular, $||X||'_{c(t_0)}=0$.
This describes some of the smoothness conditions that must be
satisfied at the endpoints.

\bigskip

We now apply this  to the $P$ family and show:
\begin{prop}\label{Pfamily}
Let $M$ be one of the $7$-manifolds $P_{ (p_-,q_-),(p_+,q_+)}$ with
its cohomogeneity one action by $G=\S^3\times \S^3$. Assume that $M$
is not one of the candidates for positive curvature $P_k$ or $P_{
(1,q),(p,1)}$. Furthermore, let $c\colon (-\infty,\infty)\to M$ be a
geodesic orthogonal to all orbits. Then for any invariant metric
with non-negative curvature there exists a Jacobi field along $c$,
given by the restriction of a Killing vector field $X^*,\; X\in\fg$,
such that $X^*$ is parallel on some interval but not for all $t$.
 In particular,
the metric is not analytic.
\end{prop}
\begin{proof} Since $H$ is finite,
we have $\fh^{\perp}=\fp\oplus\fm = \fg$. Regarding $\S^3$ as the
unit quaternions, we choose the basis of $\fg$ given by the  left
invariant vector fields $X_i$ and $Y_i$   on $G=\S^3 \times \S^3$
 corresponding to $i,j$ and $k$ in the Lie
 algebras
of the first and second $\S^3$ factor of $G$. Then the action fields
 $X_i^*,Y_i^*$ are Jacobi fields along the
  geodesic $c(t),\; \infty<t<\infty
 $ and are a basis of a self adjoint family
 $V$.

 \smallskip

 We start with three general observations.
 \smallskip

 {\bf Observation 1}.
Non-trivial irreducible representations
       of the identity component $\Ko=\S^1=\{e^{i\gt } \mid \gt \in \R \}$ consist of two dimensional
      representations given by multiplication by $e^{in\gt}$ on $\C$,
      called a weight $n$ representation. If  $\Ko = ( e^{ip\gt}, e^{iq\gt})\subset  \S^3\times\S^3$ has slope
       $(p,q)$ with $\gcd (p,q)=1 $, and $H$ is finite,  the vector space $\fp$  is given by
$\fp=\spam\{p X_1+q Y_1\}$.
 The tangent space $\fm$ to the singular orbit $G/K$ (which is spanned by the action fields $X^*$)
  splits up into $K$ irreducible subspaces
 $W_{\subo}=\spam\{X_1,Y_1\}$,
$W_1=\spam\{X_2,X_3\}$ and $W_2=\spam\{Y_2,Y_3\}$. Notice that
$W_{\subo}$ is one dimensional since $p X_1+q Y_1=0$. Thus we can
also write $W_{\subo}=\spam\{-qX_1+pY_1\}$.  The isotropy action on
$\fm$, which is given by conjugation
       on imaginary quaternions in each component,  is trivial on $W_{\subo}$ and
   has   weight $2p$ on $W_1$ and $2q$ on $W_2$ since e.g. $ e^{ip\gt}je^{-ip\gt}=e^{2ip\gt}j$. If
$p\ne q\ne 0$, all representations in $\fm$ are inequivalent and
hence orthogonal by Schur's Lemma. Furthermore, the metric on $W_i$
is a multiple of the Killing form, again by Schur's Lemma, and since
$X_i,Y_i$ are orthogonal in the Killing form, they are orthogonal in
the metric as well. Thus, unless $(p,q)=(1,1)$, the vector fields $
-q X_1+pY_1, X_2,X_3,Y_2,Y_3$ are orthogonal and $pX_1+qY_1$
vanishes.

 {\bf Observation 2}. In order to determine the derivatives $||X||'(0)$, we will use
equivariance of the second fundamental form $B\colon S^2\fm\to D$
under $\Ko$, where  $D=\R^2$ is the slice. If $H\cap \Ko=\Z_k$, then
the action of $\Ko$ on the slice has $\Z_k$ as
 its
      ineffective kernel since it acts via rotation of a circle
       and if it fixes one point, as does $H$, then it acts trivially on
       $D$. Hence the slice representation has weight $k=|H\cap
       \Ko|$. The vector space
       $S^2\fm$ splits as $S^2W_0\oplus S^2W_1\oplus S^2W_2\oplus
       W_1\otimes W_2\oplus W_0\otimes W_1\oplus W_0\otimes W_2$.
       The action of $\Ko$ on $S^2\fm$ has weight  $0$ on $S^2W_0$, $4p$ on $S^2W_1$, $4q$ on
         $S^2W_2$, and $2p\pm 2q$ on $W_1\otimes W_2$, $2p$ on $W_0\otimes W_1$ and $2q$ on $W_0\otimes W_2$.
         Thus the second fundamental form vanishes on $W_0$,  on
         $W_1$ if
         $4 |p|\ne k$, on $W_2$ if $4|q|\ne k$, on $W_1\otimes W_2$ if $|2p\pm 2q|\ne
         k$ and on
          $W_0\otimes W_1$ if $2|p|$ resp. $2|q|\ne k$.
           This
         will be used to show that in some cases
         $B(X,Y)=\ml X',Y\mr=0$ for $X\in W_i, Y\in W_j$.

\smallskip

    {\bf Observation 3}.     We will also use the  the Weyl group $W\subset N(H)/H$ of the
        cohomogeneity one action (see e.g. \cite{AA}, \cite{Z2}), which is defined as the subgroup of $G$
        which preserves the geodesic $c$.
         One easily sees that there exists a so called Weyl group element $w_-\in W$ in the normalizer of
        $H$ in $\Km=G_{c(0)}$,  unique modulo $H$, which, via the action of $G_{c(0)}$ on
        the slice $D$, satisfies $w_-(c'(0))=-c'(0)$
and hence reverses the geodesic at $t=0$. Similarly, there exists a
$w_+$ in the normalizer of $H$ in
        $\Kp=G_{c(L)}$,  unique modulo $H$, which reverses the geodesic at $t=L$.
           This
         implies that
         conjugation by $w_-$ takes the isotropy group $G_{c(rL)}$ to $G_{c(-rL)}$,
         $r\in\Z$,
         and $w_+$ takes $G_{c(rL)}$ to $G_{c(2L-rL)}$. Furthermore,
         $W$ is the dihedral group generated by $w_-$ and $w_+$. The geodesic $c$ is closed iff the Weyl group is
        finite, in which case the length of $c$ is
        $kL$ where $k$ is the order of $W$.
Finally, since $\Ko$ acts via rotation on the 2-dimensional slice,
the Weyl group element
        $w_-$ can be represented by a rotation by $\pi$ and
         hence can also be characterized as the unique element in
         $\Kmo$ which does not lie in $H$, but
        whose square lies in $H$.

        \bigskip

        We now apply these observations to the manifold $P_{
        (p_-,q_-),(p_+,q_+)}$. The Weyl group elements are given by
        $$w_-=(e^{i\frac{\pi}{4}},e^{i\frac{\pi}{4}})\in \Kmo \mod
        H,\
         \text{ and } w_+=(e^{j\frac{\pi}{4}},e^{j\frac{\pi}{4}})\in \Kpo \mod
        H$$
        since e.g. $w_-^2=(i,i)\in H$, but $w_-\notin H$.
        Notice that conjugation by $e^{i\frac{\pi}{4}} $ interchanges $j$ and $k$ and fixes $i$,
        and
        conjugation by $e^{j\frac{\pi}{4}} $ interchanges $i$ and $k$ and
        fixes $j$. Thus $w_-$ fixes $X_1$ and $Y_1$ but interchanges $X_2$ with $X_3$ and $Y_2$ with $Y_3$.
          One
        easily sees that  $W$, which is generated by
        $w_-$
        and $w_+$, has order $12$ since $(w_-w_+)^6\in H$ but $(w_-w_+)^3\notin H$. Thus
        $c$ has length $12L$. This easily implies that
        \begin{align*}G_{c(0)}&=  (e^{ip_-t},e^{iq_-t})\cdot H\; ,
   \quad G_{c(L)}=
   (e^{jp_+t},e^{jq_+t})\cdot H\; ,
   \quad
   \quad G_{c(2L)}= (e^{kp_-t},e^{kq_-t})\cdot H\\
        G_{c(3L)}&= (e^{ip_+t},e^{iq_+t})\cdot H \; ,
   \quad
        G_{c(4L)}=  (e^{jp_-t},e^{jq_-t})\cdot H
   \; ,
   \quad \ \
    G_{c(5L)}=
    (e^{kp_+t},e^{kq_+t})\cdot H
   \end{align*}
   and $G_{c(rL)}=G_{c((r-6)L)}$ for $r=6,\dots,11$.

\smallskip

    At $t=0$ we have $H\cap \Kmo=\{\pm (1,1), \pm (i,i)\}$ and hence $k=4$. The tangent space to $G/\Km$ is
    the direct sum of  $W_{\subo}=\spam\{X_1,Y_1\}=\spam\{X_1\}=\spam\{Y_1\}$ (since $p_-,q_-\ne
    0$), and
$W_1=\spam\{X_2,X_3\}$ and $W_2=\spam\{Y_2,Y_3\}$. Observation 2
implies that the second fundamental form  vanishes on $S^2(W_1)$ if
$p_-\ne 1$,
  on $S^2(W_2)$ if $q_-\ne 1$, and  on $W_1\otimes W_2$ if $2p_-+2q_-\ne
 \pm 4$, i.e. $p_-+q_-\ne
 \pm 2$. Notice that $p_--q_-=
 \pm 2$ is not possible since $p_-,q_-\equiv 1 \mod 4$ and that $p_-\ne -1$ and $q_-\ne -1$ as well.
          Similarly at $t=rL$, $r\in\Z$ since in all cases
          $k=4$.

          \smallskip

{\bf Claim 1}:  If $p_-\ne 1$ and $ p_+\ne 1$, then $X_3^*$ is a
parallel Jacobi field on $[0,L]$, but is not parallel  on $[0,2L]$.
Similarly, if $q_-\ne 1$ and $ q_+\ne 1$ for $Y_3^*$.

\smallskip

 For this we will show that $X_3^*$
satisfies all properties of \pref{rigidity2} on the
  interval $[t_0,t_1]=[0,L]$. At $t=L$ the  tangent space of $G/\Kp$
  is
    the direct sum of  $\overline{W}_{\subo}=\spam\{- q_+ X_2+p_+ Y_2\}$,
$\overline{W}_1=\spam\{X_1,X_3\}$ and
$\overline{W}_2=\spam\{Y_1,Y_3\}$. Since $X_3\in W_1\cap
\overline{W}_1$, we have
  $X_3(t)\ne 0$ for $t=0,L$, and by  Observation 2, the assumptions imply that
  $||X_3||'_t=0$ at $t=0,L$ as well. Thus condition (a) is
  satisfied. For condition (b), observe that $p_-\ne q_-$ since
  $p_-=q_-$ implies that $(p_-,q_-)=(1,1)$. Thus by Observation 1,  the vectors
   $ -q_- X_1+p_-Y_1, X_2,X_3,Y_2,Y_3$ are orthogonal at $t=0$
   and $p_- X_1+q_-Y_1$ vanishes.
  Similarly, $p_+\ne q_+$ and hence at $t=L$, the vectors $ -q_+ X_2+p_+Y_2, X_1,X_3,Y_1,Y_3$ are orthogonal
   and $p_+ X_2+q_+Y_2$ vanishes. Thus any $Z\in V$ orthogonal to
   $X_3$ at $t=0$ is also orthogonal to $X_3$ at $t=L$.
     Condition (c) holds since there are no interior
   singular points.

   Finally, we come to condition (d).
   Here we use the action of the principal isotropy group $H=
       \Delta Q$ on
         the tangent space of the regular orbits $G/H$. It acts via conjugation and thus $(i,i)$ acts via $\Id$ on
         $\spam\{X_1,Y_1\}$ and as $-\Id$ on
         $\spam\{X_2,X_3,Y_2,Y_3\}$. Similarly for $(j,j)$ and $(k,k)$.
         Hence the representation of $H$ on $\spam\{X_1,Y_1\}$, $\spam\{X_2,Y_2\}$, and $\spam\{X_3,Y_3\}$
         are inequivalent  and thus by Schur's Lemma these subspaces are  orthogonal to each other for
         all $t$. Furthermore, they are invariant under parallel translation since parallel translation
          commutes with isometries and hence  with the
         action of $H$. This implies condition (d) at $t=0$ since $Y=p_-X_1+q_-Y_1$ is
         the only element in $V$ with $Y(0)=0$
         and thus $\ml X_3'(0),Y'(0)\mr=0$.

         Altogether,    \pref{rigidity2} now implies that $X_3^*$ is parallel on $[0,L]$. On the other hand,
         $Z:=p_-X_3+q_-Y_3$ vanishes at $2L$, but $X_3^*(0)$ is not orthogonal to $Z(0)$ since $X_3(0)$ and $Y_3(0)$
          are orthogonal and $p_-\ne 0$. Hence \pref{nonparallel} implies that $X_3^*$ is not parallel on $[0,2L]$.

 \smallskip

{\bf Claim 2}: If $p_-\ne 1, q_-\ne 1$ and $p_-+ q_-\ne\pm 2$ and
$(p_-,q_-)\ne (p_+,q_+)$, then a certain linear combination of
$X^*_3$ and $Y_3^*$ is a parallel Jacobi field
 on $[0,2L]$, but not on $[0,3L]$. Similarly,
 for $p_+,q_+$.

The only Jacobi field that vanishes at $2L$ is $Z=p_-X_3+q_-Y_3$. In
order to satisfy condition (b), we choose $X=aX_3+bY_3$ such that
$\ml X(0),Z(0)\mr=0$. We will show that
 $X^*$ satisfies all properties of
\pref{rigidity2} on the
  interval $[t_0,t_1]=[0,2L]$.  Notice that at $0$ and $2L$ the
  slopes are both $(p_-,q_-)$.

  We start with condition (a). At $t=0$ we have $X\in W_1\oplus W_2$
  and hence $X\ne 0$. The assumptions on the slopes imply that the
  second fundamental form vanishes on $S^2(W_0\oplus W_1\oplus
  W_2)$, i.e. the orbit $G/\Km$ is totally geodesic. This in
  particular
  implies that
  $||X||'(0)=0$. Similarly, $||X||'(2L)=0$
  since the slopes are the same. We also have $X(2L)\ne 0$
  since the only Jacobi field vanishing at $2L$ is $Z$.
  Thus $X(2L)=0$ would contradict the orthogonality assumption at
  $t=0$.

  Condition (b) again follows from Observation 1 since $(p_-,q_-)\ne (1,1)$ implies that $p_-\ne q_-$.
  Hence the vectors
   $ -q_- X_1+p_-Y_1, X_2,X_3,Y_2,Y_3$ are orthogonal at $t=0$
   and $p_- X_1+q_-Y_1$ vanishes,
  and at $t=2L$, the vectors $ -q_- X_3+p_-Y_3, X_1,X_2,Y_1,Y_2$ are orthogonal
   and $p_- X_3+q_-Y_3$ vanishes. Since we have $\ml
   X(2L),Z(2L)\mr=0$, we chose $X$ such that $\ml
   X(0),Z(0)\mr=0$ as well. Notice also that $\ml
   X(2L),-q_- X_3(2L)+p_-Y_3(2L)\mr=0$ is not possible, since then
   $Z(2L)$ would be orthogonal to $X_3(2L)$ or $Y_3(2L)$ or both, but this is
   not possible since $a,b,p_-,q_-$ are all non-zero.

  Condition (c) holds since the only interior singularity is at $t=L$,
  and $p_+X_2+q_+Y_2$ is the only vector that vanishes there. But
  this vector is clearly orthogonal to $X$ at $t=0$.

  For condition (d) we can argue as in Claim 1.

\smallskip

Thus $X^*$ is parallel on $[0,2L]$. Finally, observe that $X(0)$ is
not orthogonal to the kernel  at $t=3L$, which is spanned by
$p_+X_3+q_+Y_3$, unless $\ml aX_3+bY_3, p_+X_3+q_+Y_3\mr_{t=0} =
ap_+||X_3||^2+bq_+||Y_3||^2=0$. Since we also have $\ml
X,p_-X_3+q_-Y_3\mr=0$, this would imply that $(p_-,q_-)=(p_+,q_+)$.
This was excluded, and thus $X^*$ is not parallel on $[0,3L]$.

\bigskip

Now we combine Claim 1 and Claim 2. Claim 1 implies that, up to
possibly switching the two $\S^3$ factors or interchanging $0$ and
$L$, we have the desired Jacobi field, unless the slopes are
$(1,q_-),(p_-,1)$ or $(p_-,q_-),(1,1)$. The first family was
excluded by assumption. In the second family we can assume that
$p_-\ne 1$, $q_-\ne 1$  and $(p_-,q_-)\ne (p_+,q_+)$, since
otherwise we are in the first family. Thus Claim 2 implies that in
the second family we have the desired Jacobi field unless $p_-+
q_-=\pm 2$. Reversing the orientation of the circle, we can assume
$p_-+ q_-= 2$.  This leaves only the candidates with slopes
$(1+2k,1-2k),(1,1)$.
         \end{proof}

\begin{rem*}
The exceptional family $P_{(1,\; q),(p,1)}$  contains several
$G$-invariant analytic metrics with non-negative curvature. Indeed,
$P_{(1,1),(-3,1)}$ is  $\Sph^7$, and $P_{(1,-3),(-3,1)}$ is the
positively curved Berger space (see e.g. \cite{GWZ} or \cite{Z2}).
It also contains $P_{(1,1),(1,1)}$. This manifold is not primitive,
and hence does not admit positive curvature. But  it does admit an
analytic metric with non-negative curvature. Indeed, we claim that
the manifold is $\Sph^3\times \Sph^4$ and that the product metric of
round sphere metrics is invariant. For this we identify the action
of $\S^3\times\S^3$ on $\Sph^3\times \Sph^4$ as
$(r_1,r_2)\in\S^3\times\S^3$ acting as $(p,q)\to (r_1p\;
r_2^{-1},\phi(r_2)q)$ where $\phi(r_2)$ acts via the well known
cohomogeneity one action of $\S^3$ on $\Sph^4$ (effectively an
$\SO(3)$ action) with group diagram $H=\{\pm 1,\pm i,\pm j,\pm
k\}\subset\{ e^{it}\cdot H\;
   ,\;
   e^{jt}\cdot H
   \}\subset\S^3$. One now easily identifies the isotropy groups of
   this action to be those of $P_{(1,1),(1,1)}$.
 \end{rem*}

        We now prove Theorem C in the Introduction.
        \smallskip
         \begin{prop}
Let $M$ be one of the $7$-manifolds $Q_{ (p_-,q_-),(p_+,q_+)}$ with
its cohomogeneity one action by $G=\S^3\times \S^3$. Assume that $M$
is not of type $Q_k=Q_{ (1,1),(k,k+1)},k\ge 0$. Furthermore, let
$c\colon (-\infty,\infty)\to M$ be a geodesic orthogonal to all
orbits. Then for any invariant metric with non-negative curvature
there exists a Jacobi field along $c$, given by the restriction of a
Killing vector field $X^*,\; X\in\fg$, such that $X^*$ is parallel
on some interval but not for all $t$.
 In particular,
the metric is not analytic.
\end{prop}
\begin{proof}
We indicate the changes that are necessary. The first difference is
the Weyl group since the Weyl group elements are now
$$w_-=(e^{i\frac{\pi}{4}},e^{i\frac{\pi}{4}})\in \Kmo \mod
        H,\
         \text{ and } w_+=(j,\pm 1)\in \Kpo \mod
        H$$
        and hence $|W|=8$, i.e. the closed geodesic has length $8L$.
        The isotropy groups are given by
        \begin{align*}G_{c(0)}&=  (e^{ip_-t},e^{iq_-t})\cdot H\; ,
   \quad G_{c(L)}=
   (e^{jp_+t},e^{jq_+t})\cdot H\; ,
   \quad
   \quad G_{c(2L)}= (e^{-ip_-t},e^{iq_-t})\cdot H\\
        G_{c(3L)}&= (e^{-kp_+t},e^{kq_+t})\cdot H \; ,
   \quad
        G_{c(4L)}=  (e^{ip_-t},e^{iq_-t})\cdot H
   \; ,
   \quad \ \
    G_{c(5L)}=
    (e^{-jp_+t},e^{jq_+t})\cdot H\\
    G_{c(6L)}&= (e^{-ip_-t},e^{iq_-t})\cdot H \; ,
   \quad
        G_{c(7L)}=  (e^{kp_+t},e^{kq_+t})\cdot H
   \; ,
   \quad \ \
    G_{c(8L)}=
    (e^{ip_-t},e^{iq_-t})\cdot H
   \end{align*}

   \smallskip

   A second difference is the normal weights. At $t=0$ we still have
   $H\cap \Kmo=\{\pm (1,1), \pm (i,i)\}$ and hence $k=4$. But at
   $t=L$ we have $H\cap \Kpo= \{(\pm 1,  1)\} $ and hence $k=2$.
   Similarly, $k=4$ at $t=2L,4L$ and $k=2$ at $t=3L,5L$. In
   particular, Observation 2 implies that $||X_3||'=||Y_3||'=0$ at
   $t=L$ and $t=3L$.

   We first claim that $(p_-,q_-)=(1,1)$. Indeed, if e.g. $p_-\ne 1$,
   then we can apply \pref{rigidity2} to $X_3$ on the
  interval $[t_0,t_1]=[0,L]$ as in the proof of Claim 1 in
  \pref{Pfamily}, since $k=2$ at $L$. For condition (b) notice that $p_+\ne q_+$ since $p_+$ is odd, and
  $q_+$ even. Furthermore, notice that  if $q_+=0$, the vectors
  $Y_1,Y_3,-q_+X_2+p_+Y_2$ do not need to be orthogonal to each
  other since $\Kpo$ acts trivially on $\overline{W}_{\subo}\oplus
  \overline{W}_2$, but they are orthogonal to $X_3\in\overline{W}_1$
  which is sufficient for condition (b).

  For condition (d) we again use the action of the principal isotropy group $H=
         \{(\pm
1, \pm 1) , (\pm i , \pm i)\}$ on
         the tangent space of the regular orbits $G/H$. $H$ acts via $\Id$ on
         $\spam\{X_1,Y_1\}$ and as $-\Id$ on
         $\spam\{X_2,X_3,Y_2,Y_3\}$. Thus by Schur's Lemma these two subspaces are orthogonal for
         all $t$ and are also invariant under parallel translation.
          This implies condition (d) since $Y=p_-X_1+q_-Y_1$ is the only element in $V$ with $Y(0)=0$
         and thus $\ml X_3',Y'\mr_{t=0}=0$. Finally, notice that
         $Z=-p_+X_3+q_+Y_3$ satisfies $Z(3L)=0$, but $\ml
         X_3(0),Z(0)\mr=p_+||X_3(0)||^2\ne 0$ and hence by
         \pref{nonparallel} $X_3^*$ is not parallel on $[0,3L]$.

\smallskip

         Next, we claim that if $p_+\pm q_+\ne\pm 1$, then we can
         argue as in the proof of Claim 2 in \pref{Pfamily}. Indeed,
         we choose $X=aX_3+bY_3$ so that $\ml X,-p_+X_3+q_+Y_3\mr=0$ at
         $t=L$ and apply  \pref{rigidity2} to $X^*$ on the interval
         $[L,3L]$. At the endpoints, the second fundamental form
         vanishes on $S^2W_i$ and $W_0\otimes W_i$ since $k=2$, and
         on $W_1\otimes W_2$ since $p_+\pm q_+\ne\pm 1$. Thus the
         singular orbits at $t=L$ and $t=3L$ are totally geodesic,
         which implies $||X^*||'=0$ at $t=L,3L$.
         The orthogonality condition on $X$ again implies condition
         (b), and for (c) we use the action of $H$ to conclude that
         $-p_-X_1+q_-Y_1$,
         the only vanishing Jacobi field at $t=2L$, is orthogonal to
         $X$ at $t=L$. For condition (d) we argue as in the previous case. Finally, notice that
         $Z=p_+X_3+q_+Y_3$ satisfies $Z(7L)=0$, but $\ml
         X(L),Z(L)\mr\ne 0$ since otherwise
         $ap_+||X_3(L)||+bq_+||Y_3(L)||^2=0$, which contradicts
          $\ml
          X(L),-p_+X_3+q_+Y_3\mr=-ap_+||X_3||^2+bq_+||Y_3||^2=0$
          since $p_+\ne 0$ and $a\ne 0$.
           Thus $X_3^*$ is not parallel on $[L,7L]$.

\smallskip

         Altogether, we can now assume that $(p_-,q_-)=(1,1)$ and
         $p_++ q_+=\pm 1$ or $p_+- q_+=\pm 1$. We can changes the
         sign of $p_+$ by conjugating all groups with $(1,j)$
         and both signs by reversing the orientation of the circle.
         Thus it is sufficient to assume $q_+-p_+=1$. But this is
         precisely the family $Q_k$ with slopes $(1,1),(k,k+1)$,
         $k\ge 0$,
         after possibly switching the two $\S^3$ factors.
\end{proof}
\smallskip

\begin{rem*}
$Q_1$ is the positively curved Aloff Wallach space which admits an
invariant analytic metric with positive curvature. It is not known
if $Q_k$ with $k>1$ admit such metrics, not even if they admit
analytic metrics with non-negative curvature.

The manifold $Q_0$ is special. In the language of our paper, any
linear combination of $Y_2$ and $Y_3$ is orthogonal to all kernels,
and hence a parallel Jacobi field for all $t$. But there is no
Jacobi field which is necessarily parallel for some $t$ but not for
all $t$. In \cite{GWZ} it was shown that $Q_0$ has the cohomology of
$\Sph^2\times\Sph^5$, but we do not know if it is diffeomorphic to
it. Furthermore, in \cite{GZ3} it was shown that it is also the
total space of the $\SO(3)$ principle bundle over $\CP^2$ with
$w_2\ne 0$ and $p_1=1$.
 \end{rem*}
\bigskip

We finally come to the proof of Theorem D. Here we consider the
cohomogeneity one manifolds with group diagram
$$
 H=  \{e\} \subset\{ \Delta \S^3
   ,\;
   (e^{ipt},e^{iqt})
   \}\subset\S^3\times\S^3,
  $$
  where $\Delta\S^3$ is embedded diagonally and $p,q$ are arbitrary
  relatively prime integers. Here we have $w_-=(-1,-1)$ and $w_+$ is
  one of $(\pm 1, \pm 1)$ and thus the normal geodesic has length
  $4L$. This implies that $G_{c(2L)}=G_{c(0)}$ and $G_{c(3L)}=G_{c(L)}$.  Here it is convenient to choose the base
  point $t_0$ to be regular in which case the Lagrange tensor
  satisfies $A_{t_0}=\Id$ and thus $X=A_tv$ with $v=X(t_0)$. $A_t$
  has two kernels, at $t=0$ and at $t=L$ (which agree with the kernels
  at $2L$ and $3L$ resp): $\ker A_0=\spam\{X_1+Y_1,X_2+Y_2,X_3+Y_3\}$ and $\ker A_L=\spam\{pX_1+qY_1\}$, all evaluated at $t_0$.
   If $(p,q)=(1,1)$, clearly $\ker A_L\subset\ker A_0$. There exists a 2-dimensional subspace $W\subset E_{t_0}$ (3-dimensional
   if $(p,q)=(1,1)$) which is orthogonal to both kernels. Thus $g_W$
   is concave for all $t$, and hence constant. By Theorem
   B, this implies that the Jacobi fields $X\in V$ with $X(t_0)\in
   W$ are parallel, and hence $R$ vanishes on this subspace. In
   particular, $R$ cannot be 2-positive. This finishes the proof of
   Theorem D

 \providecommand{\bysame}{\leavevmode\hbox
to3em{\hrulefill}\thinspace}

\end{document}